\documentclass[12pt,a4paper]{article}

\usepackage{latexsym}
\usepackage{amsmath}
\usepackage{amssymb}
\usepackage{arydshln}

\usepackage{cite}
\usepackage{stmaryrd}
\usepackage{enumerate}

\usepackage{hyperref}

\usepackage{color}
\usepackage{lineno}
\usepackage{graphicx}
\usepackage{ae}
\usepackage{amsmath}
\usepackage{amssymb}
\usepackage{latexsym}
\usepackage{url}
\usepackage{epsfig}
\usepackage{mathrsfs}
\usepackage{amsfonts}
\usepackage{amsthm}

\usepackage{float}
\usepackage{subfig}
\usepackage{tikz}
\usetikzlibrary{positioning}

\newtheorem{theorem}{Theorem}[section]

\newtheorem{prop}[theorem]{Proposition}
\newtheorem{lemma}[theorem]{Lemma}
\newtheorem{remark}{Remark}
\newtheorem{cor}[theorem]{Corollary}

\theoremstyle{definition}

\numberwithin{equation}{section} 
\allowdisplaybreaks    

\def\qed{\hfill$\Box$\vspace{12pt}}

\long\def\delete#1{}

\usetikzlibrary{decorations.markings}

\tikzstyle{vertex}=[circle, draw, inner sep=0pt, minimum size=6pt]
\tikzstyle{directed}=[postaction={decorate,
	decoration={markings,mark=at position 0.3 with {\arrow{stealth}}}
}]

\usepackage{xcolor}
\usepackage[normalem]{ulem}

\begin{document}
\title {Localization of spectral Tur{\'a}n theorems for signed graphs}

\author{Linfeng Xie$^{a,b}$,~Xiaogang Liu$^{a,b,}$\thanks{Supported by the National Natural Science Foundation of China (No. 12371358).}~$^,$\thanks{ Corresponding author. Email addresses: xielinfeng@mail.nwpu.edu.cn, xiaogliu@nwpu.edu.cn}~
	\\[2mm]
	{\small $^a$School of Mathematics and Statistics,}\\[-0.8ex]
	{\small Northwestern Polytechnical University, Xi'an, Shaanxi 710072, P.R.~China}\\
	{\small $^b$Xi'an-Budapest Joint Research Center for Combinatorics,}\\[-0.8ex]
	{\small Northwestern Polytechnical University, Xi'an, Shaanxi 710129, P.R. China}\\
}
\date{}

\openup 0.5\jot
\maketitle

\begin{abstract}
	In this paper, we extend localized Tur{\'a}n theorems to signed graphs and study the corresponding spectral Tur{\'a}n problems. Firstly, we establish a vertex-localized Tur{\'a}n-type inequality for signed graphs and characterize the extremal graphs reaching this bound. Secondly, we derive a sharp upper bound for the largest eigenvalue of a signed graph via localized Tur{\'a}n parameters, and identify the extremal graphs attaining equality. Finally, we generalize a walk-based local spectral Tur{\'a}n inequality to signed graphs. Our results generalize and improve several existing theorems for unsigned and signed graphs.

	\smallskip
	
	
\end{abstract}

\section{Introduction}
A \emph{signed graph} $\Gamma=(G,\sigma)$ consists of an unsigned graph $G=(V(G),E(G))$ and a sign function $\sigma : E(G) \to \left \{+1, -1 \right \}$, where $G$ is called the \emph{underlying graph} of $\Gamma$. In 1946, Heider \cite{Hei-JP-1946} proposed balance theory, which laid the conceptual foundation for the development of signed graphs. Based on Heider's work, Harary \cite{Har-MichM-1953} first formally founded signed graph theory, aiming to mathematically formalize the balance theory governing group relationships. After nearly three decades, Chaiken \cite{Chaiken-SIAM-1982} and Zaslavsky \cite{Zaslavsky-DAM-1982} independently obtained the matrix-tree theorem for signed graphs. For more information on signed graphs, please refer to \cite{Beck-Zaslavsky-JCTB-2006,Zaslavsky-book-2010,Zaslavsky-EJC-2018,Zaslavsky-JCTB-1987}.

Let $\Gamma=(G,\sigma)$ be a signed graph with the vertex set $V(\Gamma)=\left \{v_{1},v_{2},\dots ,v_{n}\right \}$ and the edge set $E(\Gamma)=\left \{e_{1},e_{2},\dots ,e_{m}\right \}$. Denote by $|V(\Gamma)|$ and $|E(\Gamma)|$ the order and the size of $\Gamma$, respectively. Let $N_{\Gamma}(v_i)$ denote the set of neighbors of a vertex $v_i$ in $\Gamma$ and $d_{\Gamma}(v_i)=|N_{\Gamma}(v_i)|$ the degree of $v_i$ in $\Gamma$. An edge $v_i v_j$ is called a \emph{positive edge} (respectively, \emph{negative edge}) if $\sigma(v_i v_j)=+1$ (respectively, $\sigma(v_i v_j)=-1$). The set of all positive edges (respectively, negative edges) in $\Gamma$ is denoted by $E^{+}(\Gamma)$ (respectively, $E^{-}(\Gamma)$). Denote by $(G, +)$ (respectively, $(G, -)$) the signed graph whose edges are all positive (respectively, negative). The \emph{adjacency matrix} of $\Gamma$ is defined as $A(\Gamma)=(a_{ij}^{\sigma})$, where $a_{ij}^{\sigma}=\sigma(v_{i}v_{j})a_{ij}$ and $a_{ij}=1$ if $ v_{i} $ and $v_{j}$ are adjacent, and $a_{ij}=0$ otherwise. The eigenvalues of $A(\Gamma)$ are denoted by
$$
\lambda_{1}(\Gamma)\ge \lambda_{2}(\Gamma)\ge \dots \ge \lambda_{n}(\Gamma),
$$
which are called the \emph{adjacency spectrum} of $\Gamma$. The \emph{index} of $\Gamma$ is $\lambda_{1}(\Gamma)$.

Let $\emptyset \ne U\subseteq V(G)$. The operation of changing the signs of all edges between $U$ and $V(G) \setminus U$ is called a \emph{switching}, and is also referred to as \emph{switching $\Gamma$ at a vertex subset $U$}. A signed graph $\Gamma^{\prime}$ is said to be \emph{switching equivalent} to $\Gamma$ if $ \Gamma^{\prime}$ is obtained from $\Gamma$ by a finite sequence of switchings, denoted as $\Gamma\sim \Gamma^{\prime}$. A cycle is called \emph{positive} if the number of its negative edges is even; otherwise, \emph{negative}. A signed graph is called \emph{balanced} if each of its cycles is positive; otherwise, \emph{unbalanced}.

The classical Tur{\'a}n problem is to determine the maximum number of edges in an $n$-vertex graph that does not contain a given graph $F$ as a subgraph, and this value is denoted by $\text{ex}(n,F)$. Denote by $K_n$, $C_n$ and $P_n$ the \emph{complete graph}, the \emph{cycle} and the \emph{path} of order $n$, respectively. Let $C_{\ge k}$ denote the family of cycles with at least $k$ edges, where $k\ge 3$. In 1941, Tur{\'a}n \cite{Turan-MFL-1941} first determined $\text{ex}(n,K_{r+1})$ and characterized the extremal graphs. Later, in 1959, Erd\H{o}s and Gallai \cite{Erdos-Gallai-AMASH-1959} determined $\text{ex}(n,C_{\ge k})$ and $\text{ex}(n,P_k)$, and characterized the corresponding extremal graphs, respectively. For more results on classical Tur{\'a}n problems, we refer the reader to \cite{Moon-CJM-1968, Bondy-Simonovits-JCTB-1974, Alon-Krivelevich-Sudakov-CPC-2003, Bukh-Tait-CPC-2020}.

Recently, the localized Tur{\'a}n problem has attracted considerable attention. Let $G=(V(G),E(G))$ be a simple graph. Let $c(e)$ denote the maximum value $r$ such that $e$ occurs in a subgraph of $G$ isomorphic to $K_r$. Brada\v{c} \cite{Bradac-AR-2022} and Malec and Tompkins \cite{Malec-Tompkins-EJC-2023} independently proved that
\begin{equation}\label{the first localization result for ordinary graphs}
	\sum_{e\in E(G)}\frac{c(e)}{c(e)-1}\le \frac{n^2}{2},
\end{equation}
where equality holds  in \eqref{the first localization result for ordinary graphs} if and only if $G$ is a multi-partite graph with all parts of equal size. Let $p(e)$ denote the maximum value $k$ such that $e$ occurs in a subgraph of $G$ isomorphic to $P_k$. Along this line of research, Malec and Tompkins \cite{Malec-Tompkins-EJC-2023} further proved that
\[
\sum_{e\in E(G)}\frac{1}{p(e)}\le \frac{n}{2},
\]
where the equality holds if and only if each component of $G$ is a clique.

We say that a vertex $v\in V(G)$ is a \emph{cut vertex} if its removal disconnects $G$, and a graph is \emph{$2$-connected} if it contains no cut vertices. Let $q(e)$ denote the length of a longest cycle containing $e$ if $e$ belongs to some cycle of $G$, and set $q(e)=2$ otherwise. In 2025, Zhao and Zhang \cite{Zhao-Zhang-JGT-2025} prove that
\[
\sum_{e\in E(G)}\frac{1}{q(e)}\le \frac{n-1}{2},
\]
where the equality holds if and only if $G$ is connected, and each $2$-connected component of $G$ is a clique. For more results on localized Tur{\'a}n problem, we refer the reader to \cite{Long-Ning-AR-2025,Adak-Chandran-ar-2025,Liu-Ning-AR-2025,Kan-Kumar-Pragada-AR-2025,Liu-Ning-JCTB-2026}.

Let $c(v)$ denote the order of the largest clique containing $v$ in a graph $G$. Recently, Adak and Chandran obtained the following consequence.
\begin{theorem}\emph{(See \cite[Theorem~2.1]{Adak-Chandran-ar-2025})} \label{Adak-Chandran-ar-2025}
	Let $G$ be a graph with $n$ vertices and $m$ edges. Then
	\[
	m\le\frac{n}{2}\sum_{v \in V(G)} \frac{c(v)-1}{c(v)}.
	\]
	Equality holds if and only if $G$ is either a complete regular multi-partite graph or an empty graph.
\end{theorem}

For a signed graph $\Gamma$, let $\epsilon(\Gamma)$ denote the \emph{frustration index} of $\Gamma$, which is the minimum number of edges to remove for balance, and $c_b(v)$ the order of the largest balanced clique containing $v$ for $v\in V(\Gamma)$.

In this paper, we first extend Theorem \ref{Adak-Chandran-ar-2025} to signed graphs, as stated below.

\begin{theorem}\label{signed graph version of Adak-Chandran theorem}
	Let $\Gamma=(G,\sigma)$ be a signed graph with $n$ vertices and $m~(\ge 1)$ edges. Then
	\begin{equation}\label{signed graph version of Adak-Chandran theorem-1}
		m\le \epsilon(\Gamma)+\frac{n}{2}\sum_{v \in V(\Gamma)} \frac{c_b(v)-1}{c_b(v)}.
	\end{equation}
	Equality holds if and only if $\Gamma$ is switching equivalent to a complete regular multi-partite signed graph with all edges positive by adding $\epsilon(\Gamma)$ negative edges, up to isolated vertices.
\end{theorem}

For a matrix $M \in \mathbb{R}^{m \times n}$, the \emph{Frobenius norm} of $M=(m_{ij})$ is defined as
$$
\|M\|_F := \left( \sum_{i=1}^m \sum_{j=1}^n m_{ij}^2 \right)^{1/2}.
$$
Let $B=(b_{ij})$ and $C=(c_{ij})$ be two matrices of the same dimensions. The \emph{Hadamard product} of $B$ and $C$, denoted by $B \circ C$,  is defined via element-wise multiplication: $(B \circ C)_{ij} = b_{ij}c_{ij}$. For a subset $S \subseteq [n] := \{1,2,\dots,n\}$, the \emph{characteristic vector} of $S$ is the vector $\mathbf{1}_S \in \{0,1\}^n$ whose $i$-th entry is $1$ if $i \in S$ and $0$ otherwise. Recently, Liu and Ning derived the following result.

\begin{theorem}\emph{(See \cite[Theorem~1.1]{Liu-Ning-AR-2025})}\label{Liu-Ning-AR-2025}
	Let $G$ be a weighted graph with edge weight $a(e)$ for each edge $e$ and at least one edge. Let $\lambda_1(G)$ be the largest eigenvalue of $G$. Then
	\[
	\lambda_1(G) \leq \sqrt{2 \sum_{e \in E(G)} \frac{c(e)-1}{c(e)} a(e)^2}.
	\]
	Equality holds if and only if $G$ is, up to isolated vertices, a complete $r$-partite graph for some $r$~$($with partition $V_1 \cup V_2 \cup \cdots \cup V_r$$)$, and there exists a vector $\boldsymbol{w} \in \mathbb{R}^n$ such that
	\begin{enumerate}
		\item[\rm (1)] $A(G) = \pm \sum_{i=1}^r (\mathbf{1}_{V_i} \circ \boldsymbol{w}) \big( (\mathbf{1} - \mathbf{1}_{V_i}) \circ \boldsymbol{w} \big)^\mathrm{T}$; and
		\item[\rm (2)]  $\|\mathbf{1}_{V_i} \circ \boldsymbol{w}\|^2 = \|\boldsymbol{w}\|^2 - \sqrt{1-1/r} \cdot \|A(G)\|_F$ for each $i \in [r]$.
	\end{enumerate}
\end{theorem}

A \emph{signed subgraph} is a subgraph obtained by selecting a subset of vertices and edges from a signed graph while preserving the original sign of each retained edge. Suppose that we have an infinite sequence of signed graphs $F_1, F_2, \ldots$ such that each $F_i$ is a proper signed subgraph of $F_{i+1}$. Up to switching equivalence, let $\Gamma$ be an arbitrary signed graph of order $n$, and for each edge $e \in E(\Gamma)$, define
\begin{equation}\label{localized version of extremal problems in signed graphs}
	h(e) = \max\left\{i \mid e \text{~lies~in~a~signed~subgraph~of~} \Gamma
	\text{~isomorphic~to~} F_i \right\}.
\end{equation}

The second main contribution of this paper is to generalize Theorem \ref{Liu-Ning-AR-2025} to signed graphs, as follows.

\begin{theorem}\label{Localization and the Largest Eigenvalue}
	Let $\Gamma=(G,\sigma)$ be a signed graph with at least one edge. Let $h\colon E(\Gamma)\to \mathbb{R}^+$ satisfying \eqref{localized version of extremal problems in signed graphs} and $f \colon \mathbb{R}^+ \to \mathbb{R}^+$ which is strictly decreasing. Then there exists a function set $\mathcal{H}=\left\{\alpha~|~\alpha=fh\right\}$ such that for any $\alpha \in \mathcal{H}$,
	\begin{equation}\label{upper bound for the index}
		\lambda_{1}(\Gamma) \le \sqrt{2\sum_{e\in E^{+}(\Gamma^{\prime})}\frac{1}{\alpha (e)}},
	\end{equation}
	where $\Gamma^{\prime} \sim \Gamma$ and $\lambda_{1}(\Gamma^{\prime})$ has a non-negative eigenvector. Equality holds if and only if $\Gamma$ is switching equivalent to, up to isolated vertices, a balanced complete $r$-partite graph for some $r$ $($with partition $V_1 \cup V_2 \cup \cdots \cup V_r$$)$, and there exists a vector $\boldsymbol{w} \in \mathbb{R}^n$ such that
	\begin{enumerate}
		\item[\rm (1)]  $A(\Gamma) = \pm \sum_{i=1}^r (\mathbf{1}_{V_i} \circ \boldsymbol{w}) \big( (\mathbf{1} - \mathbf{1}_{V_i}) \circ \boldsymbol{w} \big)^\mathrm{T}$; and
		\item[\rm (2)]  $\|\mathbf{1}_{V_i} \circ \boldsymbol{w}\|^2 = \|\boldsymbol{w}\|^2 - \sqrt{2|E(\Gamma)|\frac{r-1}{r}}$ for each $i \in [r]$.
	\end{enumerate}
\end{theorem}




For a graph $G=(V(G),E(G))$, a \emph{walk} with $r$ vertices is a sequence $(v_0, v_1, \dots, v_{r-1})$ satisfying $v_{i-1}v_i \in E(G)$ for $1 \leq i \leq r-1$, which has length $r-1$. Let $w_r(v)$ denote the number of walks with $r$ vertices starting at a vertex $v$ in $G$. In particular, $w_1(v)=1$ and $w_2(v)=d_{G}(v)$, which is the degree of vertex $v$ in $G$.

Let $\omega(G)$ denote the \emph{clique number} of a graph $G$, which is the order of the largest clique of $G$. Recently, Liu et al. in \cite{Liu-Sun-Wang-Wu-AR-2026} proved \cite[Conjecture~1.20]{Kan-Kumar-Pragada-AR-2025}, as follows.

\begin{theorem}\emph{(See \cite[Theorem~1.8, Theorem 4.1]{Liu-Sun-Wang-Wu-AR-2026})}\label{Liu-Sun-Wang-Wu-AR-2026}
	Let $G$ be a graph with at least one edge and $\omega(G)=\omega$, and let $r~(\ge 1)$ be an integer. Then
	\begin{equation}\label{walk and largest eigenvalue}
		\lambda_1^r(G) \le \sum_{v \in V(G)} w_r(v) \frac{c(v)-1}{c(v)}.
	\end{equation}
	Equality holds if and only if one of the following holds:
	\begin{enumerate}
		\item[\rm (i)] If $\omega>2$, then $G$ is a complete regular $\omega$-partite graph, up to isolated vertices.
		
		\item[\rm (ii)] If $\omega = 2$, then $G$ is a complete bipartite graph, up to isolated vertices; moreover, if $r$ is odd, then the nontrivial complete bipartite component is regular.
	\end{enumerate}
\end{theorem}

A walk in a signed graph $\Gamma$ is \emph{positive} if it contains an even number of negative edges; otherwise, it is \emph{negative}. Let $w_r^+(v)$ (respectively, $w_r^-(v)$) denote the number of positive (respectively, negative) walks of length $r-1$ starting at a vertex $v$ in $\Gamma$. The \emph{$r$-satisfaction vertex index} of $v$, denoted by $\zeta_r(v)$, is defined as the maximum number of positive walks of length $r-1$  starting at $v$ over all signed graphs switching equivalent to $\Gamma$. That is,
$$
\zeta_r(v)=\max_{\Gamma^{\prime} \sim \Gamma} w_r^+(v).
$$
In particular, $\zeta_1(v)=1$ and $\zeta_2(v)= d_{\Gamma}(v)$.

Let $\omega_{b}(\Gamma)$ denote the \emph{balanced clique number} of a signed graph $\Gamma$, which is the order of the largest balanced clique of $\Gamma$.

The third main contribution of this paper is to give a signed version of Theorem \ref{Liu-Sun-Wang-Wu-AR-2026}, as follows.
\begin{theorem}\label{Localization of signed walk}
	Let $\Gamma=(G,\sigma)$ be a signed graph with at least one edge and $\omega_b(\Gamma)=\omega_b$, and let $r~(\ge 1)$ be an integer. Then
	\begin{equation}\label{signed walk and largest eigenvalue}
		\lambda_1^r(\Gamma) \le \sum_{v \in V(\Gamma)}\zeta_r(v) \frac{c_b(v)-1}{c_b(v)}.
	\end{equation}
	Equality holds if and only if one of the following holds:
	\begin{enumerate}
		\item[\rm (i)] If $\omega_b>2$, then $\Gamma$ is switching equivalent to a balanced complete regular $\omega_b$-partite graph, up to isolated vertices.
		
		\item[\rm (ii)] If $\omega_b= 2$, then $\Gamma$ is switching equivalent to a balanced complete bipartite graph, up to isolated vertices; moreover, if $r$ is odd, then the nontrivial balanced complete bipartite component is regular.
	\end{enumerate}
\end{theorem}

We will prove Theorem \ref{signed graph version of Adak-Chandran theorem} in Section \ref{signed graph version of Adak-Chandran theorem-section},  Theorem \ref{Localization and the Largest Eigenvalue} in Section \ref{Prove Theorem 1.4, 1.5 and 1.6} and Theorem \ref{Localization of signed walk} in Section \ref{Localization of signed walk-section}, respectively.

\section{Proof of Theorem \ref{signed graph version of Adak-Chandran theorem}}\label{signed graph version of Adak-Chandran theorem-section}
Let $\Gamma=(G,\sigma)$ be a signed graph with underlying graph $G=(V(G),E(G))$ and $\emptyset\not=S\subseteq  E(G)$. Denote by $\Gamma-S$ the signed graph obtained from $\Gamma$ by deleting all signed edges of $\Gamma$ whose underlying edges are in $S$.

\medskip

\begin{Tproof}\textbf{~of~Theorem~\ref{signed graph version of Adak-Chandran theorem}.}
	For $\Gamma$, there exists a signed graph $\Gamma^{\prime}$ such that $\Gamma^{\prime} \sim \Gamma$ and $|E^-(\Gamma^{\prime})|=\epsilon(\Gamma)$. Let $S$ be the set of underlying edges corresponding to $E^{-}(\Gamma')$. Hence, all edges of $\Gamma'-S$ are positive. Let $c^{\prime}_b(v)$ denote the order of the largest balanced clique containing $v$ in $\Gamma^{\prime}-S$. By Theorem \ref{Adak-Chandran-ar-2025} and $c^{\prime}_b(v)\le c_b(v)$, we have
	\[
	m-\epsilon(\Gamma)\le \frac{n}{2}\sum_{v \in V(\Gamma^{\prime}-S)} \frac{c^{\prime}_b(v)-1}{c^{\prime}_b(v)}\le \frac{n}{2}\sum_{v\in V(\Gamma^{\prime})}\frac{c_b(v)-1}{c_b(v)},
	\]
	yielding \eqref{signed graph version of Adak-Chandran theorem-1}.
	
	In the following, we characterize the extremal graphs. Assume that equality holds in  \eqref{signed graph version of Adak-Chandran theorem-1}. From the proof above, we have
	\[
	m-\epsilon(\Gamma)=\frac{n}{2}\sum_{v \in V(\Gamma^{\prime}-S)} \frac{c^{\prime}_b(v)-1}{c^{\prime}_b(v)}.
	\]
	By Theorem \ref{Adak-Chandran-ar-2025}, $\Gamma^{\prime}-S$ is a complete regular multi-partite signed graph with all edges positive. Hence, $\Gamma^{\prime}$ is a complete regular multi-partite signed graph with all edges positive by adding $\epsilon(\Gamma)$ negative edges.
	
	This completes the proof. \qed
\end{Tproof}




\begin{remark}\label{Kan-Pragada-LAA-2023-m-extremal graphs}
	{\em
		Let $\Gamma=(G, \sigma)$ be a signed graph with $n$ vertices and $m$ edges. 	In 2023, Kannan and Pragada proved \cite[Theorem~3.7]{Kan-Pragada-LAA-2023} that
		\begin{equation}\label{Kan-Pragada-LAA-2023-m-equation}
			m \leq \epsilon(\Gamma)+\frac{n^2}{2}\left(1-\frac{1}{\omega_b(\Gamma)}\right).
		\end{equation}
		Note that $c_b(v)\le \omega_b(\Gamma)$ for every $v\in V(\Gamma)$. Then \eqref{signed graph version of Adak-Chandran theorem-1} implies \eqref{Kan-Pragada-LAA-2023-m-equation} immediately. Hence, Theorem \ref{signed graph version of Adak-Chandran theorem} generalizes \cite[Theorem~3.7]{Kan-Pragada-LAA-2023}.
		
		See the following example: Let $\Sigma'$ be a signed graph whose underlying graph is $K_4$ with vertex set $\left\{v_1, v_2, v_3, v_4\right\}$, and only the edge $v_1v_2$ is negative. The signed graph $\Gamma$ is obtained by adding a vertex $v_5$ to $\Sigma'$ and joining $v_5$ to $v_4$ with a positive edge. It is easy to check that $\lambda_{1}(\Gamma)$ has a non-negative eigenvector. By a simple computation,
		the bound of Theorem \ref{signed graph version of Adak-Chandran theorem} is
		$$
		\epsilon(\Gamma)+\frac{n}{2}\sum_{v \in V(\Gamma)} \frac{c_b(v)-1}{c_b(v)}= \frac{107}{12},
		$$
		and the bound of \cite[Theorem~3.7]{Kan-Pragada-LAA-2023} is
		$$
		\epsilon(\Gamma)+\frac{n^2}{2}\left(1-\frac{1}{\omega_b(\Gamma)}\right)= \frac{28}{3}.
		$$
		Clearly, $\frac{107}{12}$ is better than $\frac{28}{3}$.

		Furthermore, if equality holds in \eqref{Kan-Pragada-LAA-2023-m-equation}, then equality holds in \eqref{signed graph version of Adak-Chandran theorem-1} and $c_b(v)= \omega_b(\Gamma)$ for every $v\in V(\Gamma)$, that is, (i) $\Gamma$ is switching equivalent to a complete regular multi-partite signed graph with all edges positive by adding $\epsilon(\Gamma)$ negative edges, up to isolated vertices, and (ii) $c_b(v)= \omega_b(\Gamma)$ for every $v\in V(\Gamma)$. Note that any graph satisfying (i) also satisfies (ii). Thus, if equality holds in \eqref{signed graph version of Adak-Chandran theorem-1}, then equality holds in \eqref{Kan-Pragada-LAA-2023-m-equation}. This implies that extremal graphs attaining \eqref{signed graph version of Adak-Chandran theorem-1} and \eqref{Kan-Pragada-LAA-2023-m-equation} are the same.
	}	
\end{remark}

\section{Proof~of~Theorem~\ref{Localization and the Largest Eigenvalue}}\label{Prove Theorem 1.4, 1.5 and 1.6}

\begin{lemma}\emph{(See \cite[Lemma~2.5]{Sun-Liu-Lan-LAA-2022})}\label{Sun-Liu-Lan-LAA-2022}
	Let $\Gamma$ be a signed graph with $n$ vertices. Then there exists a signed graph $\Gamma^{\prime}$ switching equivalent to $\Gamma$ such that $\lambda_{1}(\Gamma^{\prime})$ has a non-negative eigenvector.
\end{lemma}

\begin{lemma}\emph{(See \cite[Lemma~2.1]{Hou-Tang-Wang-AMC-2019})}
	\label{Hou-Tang-Wang-AMC-2019}
	Let $\Gamma_{1}=(G,\sigma_{1})$ and $\Gamma_{2}=(G,\sigma_{2})$ be two signed graphs on the same underlying graph. The following are equivalent:
	
	$\mathrm{(1)}$~$\Gamma_{1}$ and $\Gamma_{2}$ are switching equivalent.
	
	$\mathrm{(2)}$~$A(\Gamma_{1})$ and $A(\Gamma_{2})$ are similar.
\end{lemma}

\begin{lemma}\emph{(See \cite[Lemma~2.4]{Xie-Liu-ar-2025})}\label{LanCor-S}
	If $\Gamma$ is a signed graph and $\lambda_{1}(\Gamma)$ has a non-negative eigenvector, then for any set $S\subseteq E^-(\Gamma)$, we have $\lambda_{1}(\Gamma)\le \lambda_{1}(\Gamma-S)$. Furthermore, if $\Gamma$ is unbalanced and $\Gamma-S$ is balanced, then $\lambda_{1}(\Gamma)<\lambda_{1}(\Gamma-S)$.
\end{lemma}

Let $\Gamma$ be a signed graph with $V(\Gamma)=\{1,2, \ldots,n\}$. Define
$$
S^{\pm}_{n}=\left \{(x_{1},x_{2},\dots,x_{n})^T\in \mathbb{R}^{n}\colon \sum_{i=1}^{n} |x_{i}|=1 \right\}.
$$
By \cite[Theorem 5]{Wang-Yan-Qian-LAA-2021}, we have
\begin{equation}\label{equation-clique}
	\max\left \{\mathbf{x}^{T}A(\Gamma)\mathbf{x}\colon \mathbf{x}\in S^{\pm}_{n} \right \}=1-\frac{1}{\omega_{b}(\Gamma)}.
\end{equation}

Let $h\colon E(\Gamma)\to \mathbb{R}^+$ be as defined in \eqref{localized version of extremal problems in signed graphs} and $f \colon \mathbb{R}^+ \to \mathbb{R}^+$. For any $\alpha =fh$, let $W_{\Gamma}^{\alpha}$ be an $n\times n$ matrix whose $(i,j)$-entry is $\sigma(ij)\alpha(e)$ if $e=ij\in E(\Gamma)$, and 0 otherwise. For $ \mathbf{x}=(x_{1},x_{2},\dots,x_{n})^T\in S^{\pm}_{n}$, define
\begin{align}\label{SnXBiaoDaF1}
	F_{\Gamma}^{\alpha}(\mathbf{x})& :=\mathbf{x}^{T}W_{\Gamma}^{\alpha}\mathbf{x} =2\sum_{e=ij\in E(\Gamma)}\sigma(ij)\alpha(e)x_{i}x_{j}.
\end{align}

Let $\mathrm{supp}(\mathbf{x})$ denote the \emph{support} of a vector $\mathbf{x}$, which is the set consisting of all indices corresponding to nonzero entries in $\mathbf{x}$.

\begin{lemma}\label{induces a balanced clique}
	Let $\mathbf{x}\in S^{\pm}_{n}$ be a vector such that $F_{\Gamma}^{\alpha}(\mathbf{x})=\max \left \{F_{\Gamma}^{\alpha}(\mathbf{z})\colon \mathbf{z}\in S^{\pm}_{n} \right \}$ with minimal supporting set. Then $\mathrm{supp}(\mathbf{x})$ induces a balanced clique in $\Gamma$.
\end{lemma}

\begin{proof}
	Let $\mathbf{x}=(x_{1},x_{2},\dots,x_{n})^T\in S^{\pm}_{n}$ such that $\mathbf{x}$ has a minimal supporting set and $F_{\Gamma}^{\alpha}(\mathbf{x})=\max\left \{F_{\Gamma}^{\alpha}(\mathbf{z})\colon\mathbf{z}\in S^{\pm}_{n} \right\}$. Define $S=\left \{i\colon x_{i}<0 \right\}$ and let $\Gamma^{S}$ be the signed graph obtained from $\Gamma$ by a switching at $S$. Let $\mathbf{y}=(|x_{1}|, |x_{2}|, \dots,|x_{n}|)^T\in S^{\pm}_{n}$. By \eqref{SnXBiaoDaF1}, we have $F^{\alpha}_{\Gamma^S}(\mathbf{y})=F^{\alpha}_{\Gamma}(\mathbf{x})= \max\left \{F^{\alpha}_{\Gamma^S}(\mathbf{z})\colon \mathbf{z}\in S^{\pm}_{n} \right \}$. We assert that $\mathbf{y}$ is also a vector such that $F^{\alpha }_{\Gamma^S}(\mathbf{y})=\max\left \{F^{\alpha}_{\Gamma^S}(\mathbf{z})\colon \mathbf{z}\in S^{\pm}_{n} \right \}$ with minimal supporting set. Otherwise, suppose that $\mathbf{z}_1=(a_{1},a_{2},\dots,a_{n})^T\in S^{\pm}_{n}$ such that $F^{\alpha}_{\Gamma^S}(\mathbf{z}_1)=F^{\alpha}_{\Gamma^S}(\mathbf{y})$ and $|\mathrm{supp}(\mathbf{z}_1)| < |\mathrm{supp}(\mathbf{y})|$. Let $\mathbf{z}_2=(b_{1},b_{2},\dots,b_{n})^T$ where $b_i=-a_i$ for $i\in S$ and  $b_i=a_i$ for $i \notin S$. Then
	\[
	F^{\alpha}_{\Gamma}(\mathbf{z}_2)=F^{\alpha}_{\Gamma^S}(\mathbf{z}_1)=F^{\alpha}_{\Gamma^S}(\mathbf{y})=F^{\alpha}_{\Gamma}(\mathbf{x}),
	\]
	but $|\mathrm{supp}(\mathbf{z}_2)|=|\mathrm{supp}(\mathbf{z}_1)|<|\mathrm{supp}(\mathbf{y})|=|\mathrm{supp}(\mathbf{x})|$, a contradiction.
	
	We assert that $\mathrm{supp}(\mathbf{y})$ induces an all-positive complete subgraph in $\Gamma^S$, that is, $\mathrm{supp}(\mathbf{x})$ induces a balanced clique in $\Gamma$. Otherwise, there are two vertices in $\mathrm{supp}(\mathbf{y})$, say $i$ and $j$, such that $ij\notin E(\Gamma^S)$ or $\sigma^S(ij)=-1$, where $\sigma^S$ is a sign function of $\Gamma^S$. Let $\mathbf{e}_{i}$ be the $i$-th column of the identity matrix of order $n$. Without loss of generality, assume that $\mathbf{y}^{T}W^{\alpha}_{\Gamma^S}\mathbf{e}_{i}\ge \mathbf{y}^{T}W^{\alpha}_{\Gamma^S}\mathbf{e}_{j}$. Set $\mathbf{y}'=\mathbf{y}+|x_{j}|(\mathbf{e}_{i}-\mathbf{e}_{j})$. Clearly, $\mathbf{y}' \in S^{\pm}_{n}$.
	
	If $ij\notin E(\Gamma)$, then
	$$
	F^{\alpha}_{\Gamma^S}(\mathbf{y}')-F^{\alpha}_{\Gamma^S}(\mathbf{y}) =2|x_{j}|\mathbf{y}^{T}W^{\alpha}_{\Gamma^S}(\mathbf{e}_{i}-\mathbf{e}_{j})\ge 0.
	$$
	
	If $\sigma^S(ij)=-1$, then
	$$
	F^{\alpha}_{\Gamma^S}(\mathbf{y}')-F^{\alpha}_{\Gamma^S}(\mathbf{y}) =2|x_{j}|\mathbf{y}^{T}W^{\alpha}_{\Gamma^S}(\mathbf{e}_{i}-\mathbf{e}_{j})+2x_{j}^{2}\alpha(e)>0.
	$$
	
	Note that $|\mathrm{supp}(\mathbf{y}')| < |\mathrm{supp}(\mathbf{y})|$ and $F^{\alpha}_{\Gamma^S}(\mathbf{y}')\ge F^{\alpha}_{\Gamma^S}(\mathbf{y})$ in either case, contradictions.
	
	
	This completes the proof.\qed
\end{proof}

Let $\Gamma=(G,\sigma)$ be a signed graph. For $U\subseteq V(\Gamma)$, let $\Gamma[U]$ denote the signed subgraph of $\Gamma$ induced by $U$, with edge signs inherited from $\Gamma$. Sometimes, we say that $U$ induces $\Gamma[U]$. By Lemma \ref{induces a balanced clique}, for every $\alpha$, there exists a vector set
$$
M^{\alpha}_{\Gamma}=\left\{\mathbf{x}~|~F_{\Gamma}^{\alpha}(\mathbf{x})=\max \left \{F_{\Gamma}^{\alpha}(\mathbf{z})\colon \mathbf{z}\in S^{\pm}_{n} \right \} \text{with minimal supporting set}\right\}
$$
such that $\mathrm{supp}(\mathbf{x})$ induces a balanced clique in $\Gamma$ for every $\mathbf{x}\in M^{\alpha}_{\Gamma}$.

\begin{lemma}\label{F(z)}
	Let $\alpha=fh$ such that $\alpha(e)\le \frac{|\mathrm{supp}(\mathbf{x})|}{|\mathrm{supp}(\mathbf{x})|-1}$ for every $e\in E(\Gamma[\mathrm{supp}(\mathbf{x})])$ and every $\mathbf{x}\in M^{\alpha}_{\Gamma}$. Then $F^{\alpha}_{\Gamma}(\mathbf{z})\le 1$ for each $\mathbf{z}\in S^{\pm}_{n}$.
\end{lemma}
\begin{proof}
	Let $\mathbf{t}=(t_{1},t_{2},\dots,t_{n})^T\in M^{\alpha}_{\Gamma}$. By Lemma \ref{induces a balanced clique}, $\mathrm{supp}(\mathbf{t})$ forms a balanced clique $(K,\sigma)$ in $\Gamma$ and $|K|=|\mathrm{supp}(\mathbf{t})|$. By \eqref{SnXBiaoDaF1},
	$$
	F^{\alpha}_{\Gamma}(\mathbf{t})=2\sum_{e=ij\in E((K,\sigma))}\sigma(ij)\alpha(e)t_{i}t_{j}.
	$$
	Since $(K,\sigma)$ is balanced, there exists a sign function $\theta$ such that $\Gamma^{\theta}\sim \Gamma$ and $(K,\sigma^{\theta})=(K,+)$. Assume that $\theta(i)=-1$ for $i\in A^{-}$ and $\theta(i)=+1$ for $i\in A^+$ with $A^+\cup A^-=V(\Gamma)$. Let
	$$
	\mathbf{t}^{\prime}=\left\{
	\begin{array}{l}
		t^{\prime}_{i}=-t_{i}, ~~i\in A^-, \\[0.2cm]
		t^{\prime}_{i}=t_{i}, ~~i\in A^+.
	\end{array}\right.
	$$
	By \eqref{equation-clique}, \eqref{SnXBiaoDaF1} and $\alpha(e)\le \frac{|\mathrm{supp}(\mathbf{t})|}{|\mathrm{supp}(\mathbf{t})|-1}$ for every $e\in E(\Gamma[\mathrm{supp}(\mathbf{t})])$, we obtain
	\begin{align*}
		F^{\alpha}_{\Gamma}(\mathbf{t})=F^{\alpha}_{\Gamma^{\theta}}(\mathbf{t^{\prime}})&=2\sum_{e=ij\in E((K,+))}\alpha(e)t^{\prime}_{i}t^{\prime}_{j} \\
		&\le \frac{2|K|}{|K|-1} \sum_{e=ij\in E((K,+))}|t^{\prime}_{i}||t^{\prime}_{j}|\\
		&\le \frac{2|K|}{|K|-1}\left(\frac{|K|-1}{2|K|}\right)=1.
	\end{align*}
	
	This completes the proof.\qed
\end{proof}

Lemma \ref{F(z)} implies the following result immediately.

\begin{cor}\label{the localization result for signed graphs}
	Let $\Gamma=(G,\sigma)$ be a signed graph of order $n$. Then
	\[
	\sum_{e=ij\in E(\Gamma)}\sigma(ij)\alpha(e)\le \frac{n^2}{2}.
	\]
\end{cor}

\begin{proof}
	Let $\mathbf{x}=\left\{\frac{1}{n},\frac{1}{n},\dots,\frac{1}{n}\right\}\in S^{\pm}_{n}$. By Lemma \ref{F(z)}, we have
	\[
	\frac{2}{n^2}\sum_{e=ij\in E(\Gamma)}\sigma(ij)\alpha(e)\le 1.
	\]
	This completes the proof. \qed
\end{proof}

For a signed graph $\Gamma$, let $c_{b}(e)$ denote the order of the largest balanced clique containing an edge $e$ in $\Gamma$.
\begin{remark}\label{remark1.1-3}
	{\em Set $h\colon e\mapsto c_b(e)$ and $f\colon c_b(e)\mapsto \frac{c_b(e)}{c_b(e)-1}$ for any $e\in E(\Gamma)$. Obviously, $\alpha=fh$ satisfies the conditions of Lemma \ref{F(z)}. By Corollary \ref{the localization result for signed graphs}, we have
		\[
		\sum_{e=ij\in E(\Gamma)}\sigma(ij)\frac{c_{b}(e)}{c_{b}(e)-1}\le \frac{n^2}{2}.
		\]
		Clearly, Corollary \ref{the localization result for signed graphs} yields Equation \eqref{the first localization result for ordinary graphs}.
	}
\end{remark}

Now we proceed to present the proof of Theorem \ref{Localization and the Largest Eigenvalue}.

\medskip
\begin{Tproof}\textbf{~of~Theorem~\ref{Localization and the Largest Eigenvalue}.}~Let $\alpha\in \mathcal{H}$ satisfy the conditions in Lemma \ref{F(z)}.
	By Lemmas \ref{Sun-Liu-Lan-LAA-2022} and \ref{Hou-Tang-Wang-AMC-2019}, assume that $\mathbf{z}=(z_1,z_2,\ldots, z_n)^T$ is a unit non-negative eigenvector of $\lambda_{1}(\Gamma')$. Let $\mathbf{y}=(y_1,y_2,\ldots, y_n)^T$ with $y_{i}=z_{i}^{2}$, $1\le i\le n$. Then
	\begin{align*}
		\lambda_{1}(\Gamma)&=2\sum_{e=ij\in E(\Gamma^{\prime})}\sigma(ij)z_{i}z_{j}\\
		&=2\sum_{e=ij\in E^{+}(\Gamma^{\prime})}\sqrt{y_{i}y_{j}}-2\sum_{e=ij\in E^{-}(\Gamma^{\prime})}\sqrt{y_{i}y_{j}}\\
		&\le 2\sum_{e=ij\in E^{+}(\Gamma^{\prime})}\sqrt{y_{i}y_{j}}\\
		&=2\sum_{e=ij\in E^{+}(\Gamma^{\prime})}\sqrt{\frac{1}{\alpha(e)}}\sqrt{\alpha(e)y_{i}y_{j}}.
	\end{align*}
	By Cauchy-Schwarz inequality, we have
	$$
	\lambda_{1}^{2}(\Gamma)\le 4\sum_{e\in  E^{+}(\Gamma^{\prime})}\frac{1}{\alpha(e)} \sum_{e=ij\in  E^{+}(\Gamma^{\prime})}\alpha(e)y_{i}y_{j}.
	$$
	By Lemma \ref{F(z)}, we obtain
	$$
	\lambda_{1}^{2}(\Gamma) \le 2\sum_{e\in  E^{+}(\Gamma^{\prime})}\frac{1}{\alpha(e)}.
	$$
	
	In the following, we characterize the graphs attaining the upper bounds in \eqref{upper bound for the index}. By Theorem \ref{Liu-Ning-AR-2025}, we only need to prove that the equality in \eqref{upper bound for the index} does not hold for any unbalanced signed graph. To the contrary, we assume that there exists an unbalanced signed graph $\Gamma$ with $\lambda_{1}(\Gamma)$ such that
	$$
	\lambda_{1}(\Gamma)=\sqrt{2\sum_{e\in E^{+}(\Gamma^{\prime})}\frac{1}{\alpha(e)}}.
	$$
	Then $E^{-}(\Gamma')\ne \emptyset$. Let $S$ be the set of underlying edges corresponding to $E^{-}(\Gamma')$. So $\Gamma'-S$ is balanced and then $\lambda_1(\Gamma'-S)$ has a non-negative eigenvector. By Lemma \ref{LanCor-S}, \eqref{upper bound for the index}, $\lambda_{1}(\Gamma)=\lambda_{1}(\Gamma')$, and $\lambda_{1}(\Gamma-S)=\lambda_{1}(\Gamma'-S)$, we have
	\begin{align*}
		\lambda_{1}^{2}(\Gamma) &\le  \lambda_{1}^{2}(\Gamma-S)\\
		&\le 2\sum_{e\in E^{+}(\Gamma^{\prime}-S)}\frac{1}{fw|_{E(\Gamma-S)}(e)} \\
		&\le 2\sum_{e\in E^{+}(\Gamma^{\prime}-S)}\frac{1}{fw|_{E(\Gamma)}(e)}\\
		&\le 2\sum_{e\in E^{+}(\Gamma^{\prime})}\frac{1}{fw|_{E(\Gamma)}(e)}.
	\end{align*}
	This implies that
	$$
	\lambda_{1}(\Gamma)=\lambda_{1}(\Gamma-S).
	$$
	
	Note that $\Gamma-S$ is balanced, since $\Gamma'-S$ is balanced. By Lemma \ref{LanCor-S}, we have $\lambda_{1}(\Gamma)<\lambda_{1}(\Gamma-S)$, a contradiction.
	This completes the proof. \qed
\end{Tproof}

Let $q_b(e)$ denote the length of a longest balanced cycle containing $e$ if $e$ belongs to some cycle of $\Gamma$, and set $q_b(e)=2$ otherwise. Set $h\colon e\mapsto q_b(e)$ and $f\colon q_b(e)\mapsto \frac{q_b(e)}{q_b(e)-1}$ for any $e\in E(\Gamma)$. It is easy to check that $\alpha=fh$ satisfies the conditions of Lemma \ref{F(z)} and Theorem \ref{Localization and the Largest Eigenvalue}. Then we obtain the following result.

\begin{cor}
	Let $\Gamma=(G,\sigma)$ be a signed graph of order $n$ with $\omega_{b}(\Gamma)=\omega_{b} \ge 2$. Then
	\begin{equation}\label{1}
		\lambda_{1}(\Gamma) \le \sqrt{2\sum_{e\in E^{+}(\Gamma^{\prime})}\frac{q_b(e)-1}{q_b(e)}},
	\end{equation}
	where  $\Gamma^{\prime} \sim \Gamma$ and $\lambda_{1}(\Gamma^{\prime})$ has a non-negative eigenvector. Equality holds if and only if $\Gamma$ is switching equivalent to a balanced complete bipartite graph for $\omega_{b}=2$, or a balanced complete regular $\omega_{b}$-partite graph for $\omega_{b}\ge 3$ $($possibly with some isolated vertices$)$.
\end{cor}

Similarly, set $h\colon e\mapsto c_b(e)$ and $f\colon c_b(e)\mapsto \frac{c_b(e)}{c_b(e)-1}$ for any $e\in E(\Gamma)$. Clearly, $\alpha=fh$ satisfies the conditions of Lemma \ref{F(z)} and Theorem \ref{Localization and the Largest Eigenvalue}. Hence, we have the following result.

\begin{cor}\emph{(See \cite[Theorem~1.1]{Xie-Liu-ar-2025})}\label{Xie-Liu-ar-2025-cor-1}
	Let $\Gamma=(G,\sigma)$ be a signed graph of order $n$ with $\omega_{b}(\Gamma)=\omega_{b} \ge 2$. Then
	\begin{equation}\label{1}
		\lambda_{1}(\Gamma) \le \sqrt{2\sum_{e\in E^{+}(\Gamma^{\prime})}\frac{c_{b}(e)-1}{c_{b}(e)}},
	\end{equation}
	where  $\Gamma^{\prime} \sim \Gamma$ and $\lambda_{1}(\Gamma^{\prime})$ has a non-negative eigenvector. Equality holds if and only if $\Gamma$ is switching equivalent to a balanced complete bipartite graph for $\omega_{b}=2$, or a balanced complete regular $\omega_{b}$-partite graph for $\omega_{b}\ge 3$ $($possibly with some isolated vertices$)$.
\end{cor}

Denote by
$$
d_{\Gamma}^+(v)=\left|\left\{v_j\in N_{\Gamma}(v) \mid \sigma(v v_j)=1\right\}\right|
$$
the positive degree of $v$ in a signed graph $\Gamma$.

\begin{cor}\label{positive degree and localized clique}
	Let $\Gamma=(G,\sigma)$ be a signed graph with $\omega_{b}(\Gamma)=\omega_b\ge 2$. Then
	\begin{equation}\label{vertex clique}
		\lambda_1(\Gamma)\le \sqrt{\sum_{v\in V(\Gamma)}d_{\Gamma^{\prime}}^{+}(v)\frac{c_b(v)-1}{c_b(v)}},
	\end{equation}
	where  $\Gamma^{\prime} \sim \Gamma$ and $\lambda_{1}(\Gamma^{\prime})$ has a non-negative eigenvector. Equality holds if and only if $\Gamma$ is switching equivalent to a balanced complete bipartite graph for $\omega_{b}=2$, or a balanced complete regular $\omega_{b}$-partite graph for $\omega_{b}\ge 3$ $($possibly with some isolated vertices$)$.
\end{cor}	

\begin{proof}
	By Corollary \ref{Xie-Liu-ar-2025-cor-1}, we have
	\[
	\lambda_{1}^2(\Gamma) \le 2\sum_{e\in E^{+}(\Gamma^{\prime})}\frac{c_{b}(e)-1}{c_{b}(e)}.
	\]
	Clearly, for any edge $e=uv\in E(\Gamma)$, $c_b(e)\le \min \left\{c_b(v),c_b(u)\right\}$. Hence,
	\[
	\lambda_{1}^2(\Gamma) \le \sum_{e=uv\in E^{+}(\Gamma^{\prime})}\left(\frac{c_{b}(v)-1}{c_{b}(v)}+\frac{c_{b}(u)-1}{c_{b}(u)}\right) =\sum_{v\in V(\Gamma)}d_{\Gamma^{\prime}}^{+}(v)\frac{c_b(v)-1}{c_b(v)}.
	\]
	
	Furthermore, if equality holds in \eqref{vertex clique}, then equality holds in \eqref{1} and $c_{b}(e)=c_{b}(v)$ for any $e\in E(\Gamma)$ and $v\in V(\Gamma)$, that is, (i) $\Gamma$ is switching equivalent to a balanced complete bipartite graph for $\omega_{b}=2$, or a balanced complete regular $\omega_{b}$-partite graph for $\omega_{b}\ge 3$ $($possibly with some isolated vertices$)$, and (ii) $c_b(e)=c_b(v)$ for any $e\in E(\Gamma)$ and $v\in V(\Gamma)$. Note that any graph satisfying (i)  also satisfies (ii). Thus, if equality holds in \eqref{1}, then equality holds in \eqref{vertex clique}. This implies that extremal graphs attaining \eqref{1} and \eqref{vertex clique} are the same.
	
	This completes the proof. \qed
\end{proof}

\begin{remark}\label{remark1.1-2}
	{\em (a)~Let $G$ be a graph with $\omega(G)=\omega$. Recently, Liu and Ning prove \cite[Corollary~2.6]{Liu-Ning-AR-2025} that
		\[
		\lambda_1(G)\le \sqrt{\sum_{v\in V(G)}d_{G}(v)\frac{c(v)-1}{c(v)}}.
		\]
		Equality holds if and only if $G$ is a complete bipartite graph for $\omega=2$, or a   complete regular $\omega$-partite graph for $\omega\ge 3$ (possibly with some isolated vertices). Clearly, Corollary \ref{positive degree and localized clique} is the signed graph version of the above result.
		
		(b)~Let $\Gamma$ be a signed graph with $m$ edges and frustration index $\epsilon(\Gamma)$. It is known \cite[Lemma~3.1]{Kan-Pragada-LAA-2023} that if $\Gamma \sim \Gamma^{\prime}$, then $|E^{+}(\Gamma^{\prime})|\le m-\epsilon(\Gamma)$. By Corollary \ref{positive degree and localized clique} and the fact that $c_{b}(v)\le \omega_{b}(\Gamma)$ for any $v\in V(\Gamma)$, we have
		\[
		\lambda_{1}^{2}(\Gamma)\le \sum_{v\in V(\Gamma)}d_{\Gamma^{\prime}}^{+}(v)\frac{\omega_{b}(\Gamma)-1}{\omega_{b}(\Gamma)} =2|E^{+}(\Gamma^{\prime})|\frac{\omega_{b}(\Gamma)-1}{\omega_{b}(\Gamma)} \le 2(m-\epsilon(\Gamma))\frac{\omega_{b}(\Gamma)-1}{\omega_{b}(\Gamma)},
		\]
		that is,
		\begin{equation}\label{Kan-Pragada-LAA-2023-m-equation-1}
			\lambda_1(\Gamma) \leq \sqrt{2\bigl(m - \epsilon(\Gamma)\bigr)\left(1 - \frac{1}{\omega_b(\Gamma)}\right)}.
		\end{equation}
		This result was also established by Kannan and Pragada in 2023 (See \cite[Theorem~3.3]{Kan-Pragada-LAA-2023}).
		
		Furthermore, if equality holds in \eqref{Kan-Pragada-LAA-2023-m-equation-1}, then equality holds in \eqref{vertex clique} and $c_{b}(e)=c_{b}(v)$ for any $e\in E(\Gamma)$ and $v\in V(\Gamma)$, that is, (i) $\Gamma$ is switching equivalent to a balanced complete bipartite graph for $\omega_{b}=2$, or a balanced complete regular $\omega_{b}$-partite graph for $\omega_{b}\ge 3$ $($possibly with some isolated vertices$)$, (ii) $c_b(v)= \omega_b(\Gamma)$ for every $v\in V(\Gamma)$, and (iii) $|E^{+}(\Gamma^{\prime})|= m-\epsilon(\Gamma)$. Note that any graph satisfying (i) also satisfies (ii) and (iii). Thus, if equality holds in \eqref{vertex clique}, then equality holds in \eqref{Kan-Pragada-LAA-2023-m-equation-1}. This implies that extremal graphs attaining \eqref{vertex clique} and \eqref{Kan-Pragada-LAA-2023-m-equation-1} are the same.

		It is worth noting that the equality condition of \eqref{Kan-Pragada-LAA-2023-m-equation-1} is also given in \cite{Lan-Li-Liu-BMMS-2023}.
	}
\end{remark}

\section{Proof of Theorem \ref{Localization of signed walk}}\label{Localization of signed walk-section}

\begin{Tproof}\textbf{~of~Theorem~\ref{Localization of signed walk}.}
	By Lemma \ref{Sun-Liu-Lan-LAA-2022}, there exists a signed graph $\Gamma^{\prime}$ such that $\Gamma^{\prime}\sim \Gamma$ and $\lambda_{1}(\Gamma^{\prime})$ has a non-negative eigenvector. Denote by $S$ the set of underlying edges corresponding to $E^{-}(\Gamma')$. Let $w_r^{\prime}(v)$ be the number of walks with $r$ vertices starting at $v$ in $\Gamma^{\prime}-S$ and let $c^{\prime}(v)$ be the order of the largest clique containing $v$ in $\Gamma^{\prime}-S$. By Lemma \ref{LanCor-S} and Theorem \ref{Liu-Sun-Wang-Wu-AR-2026}, we have
	$$
	\lambda^r_1(\Gamma) \le \lambda^r_1(\Gamma^{\prime}-S)\le \sum_{v \in V(\Gamma^{\prime}-S)} w_r^{\prime}(v) \frac{c^{\prime}(v)-1}{c^{\prime}(v)}.
	$$
	Note that $w_r^{\prime}(v)\le \zeta_r(v)$ and $c^{\prime}(v)\le c_b(v)$. Thus,
	$$
	\lambda^r_1(\Gamma) \le \sum_{v \in V(\Gamma^{\prime}-S)}\zeta_r(v) \frac{ c_b(v)-1}{ c_b(v)}\le \sum_{v \in V(\Gamma)} \zeta_r(v) \frac{c_b(v)-1}{c_b(v)}.
	$$
	
	In the following, we characterize the graphs attaining the upper bounds in \eqref{signed walk and largest eigenvalue}. By the equality in \eqref{walk and largest eigenvalue}, we only need to prove that the equality in \eqref{signed walk and largest eigenvalue} does not hold for any unbalanced signed graph. To the contrary, assume that there exists an unbalanced signed graph $\Gamma$ such that
	$$
	\lambda^r_1(\Gamma) = \sum_{v \in V(\Gamma)} \zeta_r(v) \frac{c_b(v)-1}{c_b(v)}.
	$$
	Then $E^{-}(\Gamma')\ne \emptyset$. Let $M$ be the set of underlying edges corresponding to $E^{-}(\Gamma')$. So $\Gamma'-M$ is balanced and then $\lambda_1(\Gamma'-M)$ has a non-negative eigenvector. By Lemma \ref{LanCor-S}, \eqref{signed walk and largest eigenvalue}, $\lambda_{1}(\Gamma)=\lambda_{1}(\Gamma')$, and $\lambda_{1}(\Gamma-M)=\lambda_{1}(\Gamma'-M)$, we have
	$$
	\lambda_{1}^r(\Gamma)\le \lambda_{1}^r(\Gamma-M)\le \sum_{v \in V(\Gamma^{\prime}-M)} \zeta^{\prime}_r(v) \frac{c''_b(v)-1}{c''_b(v)},
	$$
	where $\zeta^{\prime}_r(v)$ is the $r$-satisfaction vertex index of $v$ in $\Gamma^{\prime}-M$ and $c''_b(v)$ is the order of the largest clique containing $v$ in $\Gamma^{\prime}-M$. Note that $\zeta^{\prime}_r(v)\le\zeta_r(v)$ and $c''_b(v)\le c_b(v)$. We have
	$$
	\sum_{v \in V(\Gamma^{\prime}-M)} \zeta^{\prime}_r(v) \frac{c''_b(v)-1}{c''_b(v)}\le \sum_{v \in V(\Gamma)}\zeta_r(v) \frac{c_b(v)-1}{c_b(v)}.
	$$
	This implies that
	$$
	\lambda_{1}(\Gamma)=\lambda_{1}(\Gamma-M).
	$$
	
	Note that $\Gamma-M$ is balanced, since $\Gamma'-M$ is balanced. By Lemma \ref{LanCor-S}, we have $\lambda_{1}(\Gamma)<\lambda_{1}(\Gamma-M)$, a contradiction.
	
	This completes the proof. \qed
\end{Tproof}


Let $r=1$. Theorem \ref{Localization of signed walk} implies the following result immediately.

\begin{cor}\label{local refinement of signed graph version of Wilf's theorem}
	Let $\Gamma=(G,\sigma)$ be a signed graph of order $n$. Then
	\begin{equation}\label{Localization of signed graph version of Wilf's theorem}
		\lambda_{1}(\Gamma) \le \sum_{v\in V(\Gamma)}\frac{c_b(v)-1}{c_b(v)}.
	\end{equation}
	Equality holds if and only if $\Gamma$ is switching equivalent to, up to isolated vertices, a balanced complete regular multi-partite graph.
\end{cor}

\begin{remark}\label{Rem-for-Wang-Yan-Qian-LAA-2021}
	{\em Let $\Gamma$ be a signed graph with $n$ vertices.  In 2021, Wang, Yan and Qian proved \cite[Proposition~5]{Wang-Yan-Qian-LAA-2021} that
		\begin{equation}\label{a signed version of Wilf' theorem}
			\lambda_{1}(\Gamma) \le n\left(1-\frac{1}{\omega_b(\Gamma)}\right),
		\end{equation}
		which is the signed version of Wilf's theorem. Recall that $c_b(v)\le \omega_b(\Gamma)$ for every $v\in V(\Gamma)$.  Then \eqref{Localization of signed graph version of Wilf's theorem} implies \eqref{a signed version of Wilf' theorem} immediately.
		
		Let $\Sigma'$ be as in Remark \ref{Kan-Pragada-LAA-2023-m-extremal graphs}. By a simple calculation, we have
		$$
		\sum_{v\in V(\Sigma')}\frac{c_b(v)-1}{c_b(v)}=\frac{19}{6}
		$$
		and
		$$
		n\left(1-\frac{1}{\omega_b(\Sigma')}\right)=\frac{10}{3}.
		$$
		Clearly, $\frac{19}{6}$ is better than $\frac{10}{3}$.
		
		Furthermore, if equality holds in \eqref{a signed version of Wilf' theorem}, then equality holds in \eqref{Localization of signed graph version of Wilf's theorem} and $c_b(v)= \omega_b(\Gamma)$ for every $v\in V(\Gamma)$, that is, (i) $\Gamma$ is switching equivalent to, up to isolated vertices, a balanced complete regular multi-partite graph, and (ii) $c_b(v)= \omega_b(\Gamma)$ for every $v\in V(\Gamma)$. It is clear that any graph satisfying (i) also satisfies (ii). Thus, if equality holds in \eqref{Localization of signed graph version of Wilf's theorem}, then equality holds in \eqref{a signed version of Wilf' theorem}. This implies that extremal graphs attaining \eqref{a signed version of Wilf' theorem} and \eqref{Localization of signed graph version of Wilf's theorem} are the same.
		
		It is worth noting that the equality condition of \eqref{a signed version of Wilf' theorem} is also given in \cite{Lan-Li-Liu-BMMS-2023}.
	}
\end{remark}
Let $r=2$. Theorem \ref{Localization of signed walk} implies a weaker version of Corollary \ref{positive degree and localized clique}.
\begin{cor}
	Let $\Gamma=(G,\sigma)$ be a signed graph with $\omega_{b}(\Gamma)=\omega_b$. Then
	\begin{equation}
		\lambda_1(\Gamma)\le \sqrt{\sum_{v\in V(\Gamma)}d_{\Gamma}(v)\frac{c_b(v)-1}{c_b(v)}}.
	\end{equation}
	Equality holds if and only if $\Gamma$ is switching equivalent to a balanced complete bipartite graph for $\omega_{b}=2$, or a balanced complete regular $\omega_{b}$-partite graph for $\omega_{b}\ge 3$ $($possibly with some isolated vertices$)$.
\end{cor}	

At the end of this section, we give an alternative proof of Corollary \ref{local refinement of signed graph version of Wilf's theorem}.

Let $G$ be a graph with $V(G)=\{1,2,\ldots,n\}$ and no isolated vertices. Let $A(G)$ denote the adjacency matrix of $G$, and let $W(G)=(w_{ij})$ denote the $n \times n$ symmetric matrix with $w_{ij}=\frac{1}{2}\left(\frac{c(i)}{c(i)-1}+\frac{c(j)}{c(j)-1}\right)$, where $c(i)$ denotes the order of the largest clique containing vertex $i$ in $G$. Define
$$
S^{+}_{n}=\left \{(x_{1},x_{2},\dots,x_{n})^T\in \mathbb{R}^{n}\colon x_i \geq 0,\ i \in \{1,2,\ldots,n\}\ \text{and}\ \sum_{i=1}^n x_i = 1 \right\}.
$$

\begin{lemma}\emph{(See \cite[Lemma~3.2]{Liu-Ning-AR-2025})}\label{Liu-Ning-AR-2025-lemma-1}
	Let $G$ be a graph of order $n$ with $\omega(G) = \omega$ and no isolated vertices, and let
	$\mathbf{x} \in S^{+}_{n}$. Then
	\[
	\mathbf{x}^{\mathrm{T}} \bigl(W(G) \circ A(G)\bigr) \mathbf{x} \leq 1.
	\]
	Equality holds if and only if the induced subgraph of $G$ on $\mathrm{supp}(\mathbf{x})$ is a complete $\omega$-partite graph
	whose vertex classes $V_1, V_2, \dots, V_\omega$ satisfy $\sum_{v \in V_i} x_v = \frac{1}{\omega}$ for all $i \in \{1,2,\ldots, \omega\}$.
\end{lemma}


\begin{prop}\label{signed degree and localized clique}
	Let $\Gamma=(G,\sigma)$ be a signed graph. Then
	\begin{equation}\label{signed degree and localized clique-1}
		\sum_{v\in V(\Gamma)}d_{\Gamma^{\prime}}^{+}(v)\frac{c_b(v)-1}{c_b(v)}\le \left(\sum_{v\in V(\Gamma)}\frac{c_b(v)-1}{c_b(v)}\right)^2,
	\end{equation}
	where $\Gamma^{\prime}$ is a signed graph such that $\Gamma^{\prime}\sim \Gamma$. 	Equality holds if and only if $\Gamma$ is switching equivalent to a complete regular multi-partite signed graph with all edges positive by adding $|E^{-}(\Gamma^{\prime})|$ negative edges, up to isolated vertices.
\end{prop}

\begin{proof}	
	Let
	\[
	\mathbf{x}=\left(\frac{c_b(v_1)-1}{c_b(v_1)}, \frac{c_b(v_2)-1}{c_b(v_2)}, \dots, \frac{c_b(v_n)-1}{c_b(v_n)}\right)^T~\text{for $v_i\in V(\Gamma)$},~1\le i \le n,
	\]
	and
	$$
	\left\|\mathbf{x}\right\|_1= \sum_{i=1}^n\frac{c_b(v_i)-1}{c_b(v_i)} .
	$$
	Then $\frac{\mathbf{x}}{\left\|\mathbf{x}\right\|_1}\in S^{+}_{n}$.
	
	
	For any $\Gamma^{\prime} \sim \Gamma$, let $S$ be the set of underlying edges corresponding to $E^{-}(\Gamma')$, and let $G-S$ be the underling graph of $\Gamma^{\prime}-S$. Denote by $c_b^{\prime}(v)$ the order of the largest clique containing vertex $v$ in $\Gamma^{\prime}-S$ and $c^{\prime}(v)$ the order of the largest clique containing vertex $v$ in $G-S$. Note that $\Gamma^{\prime}-S$ is a signed graph with all edges positive. Hence, $c_b^{\prime}(v)=c^{\prime}(v)$. By Lemma \ref{Liu-Ning-AR-2025-lemma-1} and $c_b^{\prime}(v)\le c_b(v)$, we have
	\begin{align*}
		\sum_{v\in V(\Gamma)}d_{\Gamma^{\prime}}^{+}(v)\frac{c_b(v)-1}{c_b(v)}&=\sum_{uv\in E^+(\Gamma^{\prime})}\left(\frac{c_b(u)-1}{c_b(u)}+\frac{c_b(v)-1}{c_b(v)}\right)\\
		&=\sum_{uv\in E^+(\Gamma^{\prime})}\left(\frac{c_b(u)}{c_b(u)-1}+\frac{c_b(v)}{c_b(v)-1}\right)\frac{(c_b(u)-1)(c_b(v)-1)}{c_b(u)c_b(v)}\\
		&\le \sum_{uv\in E^+(\Gamma^{\prime})}\left(\frac{c_b^{\prime}(u)}{c_b^{\prime}(u)-1}+\frac{c_b^{\prime}(v)}{c_b^{\prime}(v)-1}\right)\frac{(c_b(u)-1)(c_b(v)-1)}{c_b(u)c_b(v)}\\
		&=\sum_{uv\in E^+(\Gamma^{\prime})}\left(\frac{c^{\prime}(u)}{c^{\prime}(u)-1}+\frac{c^{\prime}(v)}{c^{\prime}(v)-1}\right)\frac{(c_b(u)-1)(c_b(v)-1)}{c_b(u)c_b(v)}\\
		&= \left(\frac{\mathbf{x}}{\left\|\mathbf{x}\right\|_1}\right)^T  \bigl(W(G-S) \circ A(G-S)\bigr) \left(\frac{\mathbf{x}}{\left\|\mathbf{x}\right\|_1}\right) \left\|\mathbf{x}\right\|^2_1 \\
		&\le \left\|\mathbf{x}\right\|^2_1 =\left(\sum_{v\in V(\Gamma)}\frac{c_b(v)-1}{c_b(v)}\right)^2.
	\end{align*}

	Finally, we characterize the graphs attaining the upper bounds in \eqref{signed degree and localized clique-1}. From the proof above, we know that equality holds if and only if $c_b^{\prime}(v)= c_b(v)$ and
	$$
	\left(\frac{\mathbf{x}}{\left\|\mathbf{x}\right\|_1}\right)^T  \bigl(W(G-S) \circ A(G-S)\bigr) \left(\frac{\mathbf{x}}{\left\|\mathbf{x}\right\|_1}\right)=1.
	$$
	Let $\omega(G-S)=\omega'$. Hence, by Lemma \ref{Liu-Ning-AR-2025-lemma-1}, we obtain that $\mathrm{supp}(\mathbf{x})$ is a complete $\omega'$-partite graph whose vertex classes $V_{1}, V_{2}, \ldots, V_{\omega'}$ satisfy
	\[
	\left\|\mathbf{1}_{V_{i}} \circ \mathbf{x}\right\|_{1}=\frac{1}{\omega'}
	\]
	for all $i \in\left\{1,2,\ldots,\omega'\right\}$. Note that
	\[
	\frac{1}{\omega'}=\left\|\mathbf{1}_{V_{i}} \circ \mathbf{x}\right\|_{1}=\frac{1}{\|\mathbf{x}\|_{1}} \sum_{v \in V_{i}} \frac{c^{\prime}(v)-1}{c^{\prime}(v)}=\left(1-\frac{1}{\omega'}\right) \frac{\left|V_{i}\right|}{\|\mathbf{x}\|_{1}},
	\]
	for any $i \in\left\{1,2,\ldots,\omega'\right\}$.
	Then we have
	\[
	\left|V_{1}\right|=\left|V_{2}\right|=\dots=\left|V_{\omega'}\right| =\frac{\|\mathbf{x}\|_{1}}{\omega'-1}.
	\]
	Hence, $\Gamma^{\prime}-S=(G-S,+)$ is, up to isolated vertices, a complete regular multi-partite graph of all positive edges. Hence, $\Gamma^{\prime}$ is obtained from a complete regular multi-partite signed graph with all edges positive by adding $|S|$ negative edges, up to isolated vertices. At this point, we observe that $\Gamma^{\prime}$ satisfies $c_b^{\prime}(v)= c_b(v)$ for any $v\in V(\Gamma)$.
	
	This completes the proof. \qed
\end{proof}

\begin{Tproof}\textbf{~of~Corollary~\ref{local refinement of signed graph version of Wilf's theorem}.}~By Proposition  \ref{signed degree and localized clique} and Corollary \ref{positive degree and localized clique}, we have
	\[
	\lambda_1(\Gamma)\le \sqrt{\sum_{v\in V(\Gamma)}d_{\Gamma^{\prime}}^{+}(v)\frac{c_b(v)-1}{c_b(v)}}\le\sum_{v\in V(\Gamma)}\frac{c_b(v)-1}{c_b(v)},
	\]
	yielding \eqref{local refinement of signed graph version of Wilf's theorem}.
	
	Furthermore, if equality holds in \eqref{Localization of signed graph version of Wilf's theorem}, then equalities must hold in both \eqref{signed degree and localized clique-1} and \eqref{vertex clique}. Clearly, equality holds in \eqref{Localization of signed graph version of Wilf's theorem} if and only if $\Gamma$ is switching equivalent to, up to isolated vertices, a balanced complete regular multi-partite graph.
	
	This completes the proof. \qed
\end{Tproof}

\begin{remark}
	{\em
		Suppose that $\Gamma$ is a connected signed graph. Let $2\le t\le \omega_b(\Gamma)$ and
		\[
		f(t)=\left(C_1 +d_{\Gamma^{\prime}}^{+}(u)\left(1-\frac{1}{t}\right)\right)-\left(C_2 +\left(1-\frac{1}{t}\right)\right)^2,
		\]
		where $C_1=\sum_{v\in V(\Gamma-u)}d_{\Gamma^{\prime}}^{+}(v)\frac{c_b(v)-1}{c_b(v)}$ and $C_2=\sum_{v\in V(\Gamma-u)}\frac{c_b(v)-1}{c_b(v)}$. By a simple computation, we have
		\[
		f^{\prime}(t)=\frac{d_{\Gamma^{\prime}}^{+}(u)+\frac{2}{t}-2(C_2+1)}{t^2}.
		\]
		Note that $2c_b(v)-2\ge c_b(v)$ for any $v\in V(\Gamma)$. Then
		\begin{align*}
			d_{\Gamma^{\prime}}^{+}(u)+\frac{2}{t}-2(C_2+1)&\le d_{\Gamma^{\prime}}^{+}(u)-2C_2 -1\\
			&=d_{\Gamma^{\prime}}^{+}(u)-\sum_{v\in V(\Gamma-u)}\frac{2c_b(v)-2}{c_b(v)} -1\\
			&\le d_{\Gamma^{\prime}}^{+}(u)-(n-1)-1\\
			&=d_{\Gamma^{\prime}}^{+}(u)-n\\
			&<0.
		\end{align*}
		Thus, $f(t)$ is monotonically decreasing with respect to $t$.
		
		Since $2\le c_b(v)\le \omega_b(\Gamma)$ for all $v\in V(\Gamma)$, we have
		\begin{align*}
			0&\ge \left(\sum_{v\in V(\Gamma)}d_{\Gamma^{\prime}}^{+}(v)\frac{c_b(v)-1}{c_b(v)}\right)-\left(\sum_{v\in V(\Gamma)}\frac{c_b(v)-1}{c_b(v)}\right)^2 \\
			& \ge \left(\sum_{v\in V(\Gamma)}d_{\Gamma^{\prime}}^{+}(v)\frac{\omega_b(\Gamma)-1}{\omega_b(\Gamma)}\right) -\left(\sum_{v\in V(\Gamma)}\frac{\omega_b(\Gamma)-1}{\omega_b(\Gamma)}\right)^2 \\
			&=2|E^{+}(\Gamma^{\prime})|\left(\frac{\omega_b(\Gamma)-1}{\omega_b(\Gamma)} \right) -n^2\left(\frac{\omega_b(\Gamma)-1}{\omega_b(\Gamma)}\right)^2,
		\end{align*}
		that is,
		\[
		|E^{+}(\Gamma^{\prime})|\le \frac{n^2}{2} \left(1-\frac{1}{\omega_b(\Gamma)}\right).
		\]
		Obviously, we can always choose a signed graph $\Gamma^{\prime}$ such that $\Gamma^{\prime}\sim \Gamma$ and $|E^{+}(\Gamma^{\prime})|=m-\epsilon(\Gamma)$.
		
		If $\Gamma$ is disconnected, we may proceed similarly by considering each connected component individually.
		
		Therefore, Proposition \ref{signed degree and localized clique} implies \eqref{Kan-Pragada-LAA-2023-m-equation} (See \cite[Theorem~3.7]{Kan-Pragada-LAA-2023}).
	}
\end{remark}

\end{document}